\documentclass[12pt, a4paper]{amsart}
\usepackage[hmargin=30mm, vmargin=25mm, includefoot, twoside]{geometry}
\usepackage[bookmarksopen=true]{hyperref}
\usepackage[latin1]{inputenc}  
\usepackage[french,english]{babel}
\usepackage{amssymb,verbatim,  pstricks, pst-plot, textcomp, fontenc, euscript}
\usepackage[alphabetic]{amsrefs}
\usepackage{latexsym}
\usepackage{mathrsfs}
\usepackage{xspace}
\usepackage{enumerate, paralist}
\newtheorem{prop}{Proposition}
\newtheorem{thm}[prop]{Theorem}
\newtheorem*{thm*}{Theorem}
\newtheorem*{prop*}{Proposition}

\newtheorem*{addendum*}{Addendum}
\newtheorem{cor}[prop]{Corollary}
\newtheorem{lem}[prop]{Lemma}

\newtheorem*{convention*}{Convention}
\theoremstyle{definition}
\newtheorem*{defn*}{Definition}

\newtheorem{remark}[prop]{Remark}

\newtheorem*{scholium*}{Scholium}
\theoremstyle{remark}

\newtheorem*{example*}{Example}
\numberwithin{equation}{section}

\newcommand{\FF}{\mathbf{F}}

\newcommand{\ZZ}{\mathbf{Z}}

\newcommand{\la}{\langle}
\newcommand{\ra}{\rangle}

\newcommand{\tangle}[2]
{\angle_\mathrm{T}(#1,#2)}
\newcommand{\aangle}[3]
{\angle_{#1}(#2,#3)}
\newcommand{\cangle}[3]
{\overline{\angle}_{#1}(#2,#3)}

 \DeclareMathOperator{\Ker}{Ker}

\def\Aut{\mathop{\mathrm{Aut}}\nolimits}

\def\max{\mathop{\mathrm{max}}\nolimits}

\begin{document}

\title[Simplicity of some twin tree lattices]{Simplicity of twin tree lattices\\ with non-trivial commutation relations}

\author[P-E.~Caprace]{Pierre-Emmanuel Caprace*}
\address{Universit\'e catholique de Louvain, IRMP, Chemin du Cyclotron 2, 1348 Louvain-la-Neuve, Belgium}
\email{pe.caprace@uclouvain.be}
\thanks{*F.R.S.-FNRS Research Associate, supported in part by FNRS grant F.4520.11 and the European Research Council}

\author[B. R\'emy]{Bertrand R\'emy**}
\address{Universit\' e Lyon 1\\
Institut Camille Jordan\\
UMR 5208 du CNRS\\
43 blvd du 11 novembre 1918\\
F-69622 Villeurbanne Cedex, France}
\email{remy@math.univ-lyon1.fr}
\thanks{**Supported in part by Institut Universitaire de France and Institut Mittag-Leffler}

\date{\today}

\maketitle

\begin{abstract}
We prove a simplicity criterion for certain twin tree lattices.
It applies to all rank two Kac--Moody groups over finite fields with non-trivial commutation relations, thereby yielding  examples of simple non-uniform lattices in the product of two trees. 
\end{abstract}

\bigskip

This paper deals with the construction of finitely generated (but not finitely presented) simple groups acting  as non-uniform lattices on products of two twinned trees.  These seem to be the first examples of \emph{non-uniform} simple lattices in the product of two trees. They contrast with the simple groups obtained by M.~Burger and Sh.~Mozes \cite{BurgerMozes}~in a similar geometric context: the latter groups are (torsion-free) \emph{uniform} lattices, in the product of two trees;  in particular they are finitely presented. That a non-uniform lattice in a $2$-dimensional CAT($0$) cell complex cannot be finitely presented is a general fact recently proved by G.~Gandini~\cite[Cor.~3.6]{Gandini}. 

The lattices concerned by our criterion belong to the class of \textbf{twin building lattices}. By definition, a twin building lattice is a special instance of a group endowed with a \textbf{Root Group Datum} (also sometimes called \textbf{twin group datum}), i.e. a group $\Lambda$ equipped with a family of subgroups $(U_\alpha)_{\alpha \in \Phi}$, called \textbf{root subgroups}, indexed by the (real) roots of some root system $\Phi$ with Weyl group $W$, and satisfying a few conditions called the \textbf{RGD-axioms},    see \cite{TitsTwin} and \cite{CaRe}. Such a group $\Lambda$ acts by automorphism on a product of two buildings $X_+ \times X_-$,  preserving a twinning between $X_+$ and $X_-$. The main examples arise from Kac--Moody theory, see \cite{TitsKM, TitsTwin}. When the root groups are finite, the group $\Lambda$ is finitely generated,  the buildings $X_+$ and $X_-$ are locally finite and the $\Lambda$-action on $X_+ \times X_-$ is properly discontinuous. In particular  (modulo the finite kernel) $\Lambda$ can be viewed as a discrete subgroup of the locally compact group $\Aut(X_+) \times \Aut(X_-)$. The quotient $\Lambda \backslash \Aut(X_+) \times \Aut(X_-)$ is never compact. However, if in addition the order of each root group
 is at least as large as the rank of the root system $\Phi$, then $\Lambda$ has finite covolume; in particular $\Lambda$ is a non-uniform lattice in  $\Aut(X_+) \times \Aut(X_-)$, see \cite{RemCRAS} and \cite{CaRe}. When  $\Lambda$ has finite covolume in $\Aut(X_+) \times \Aut(X_-)$, it is called a  \textbf{twin building lattice}.

It was proved in \cite{CaRe} that a twin building lattice is infinite and virtually simple provided 
the associated Weyl group $W$ is irreducible and not virtually abelian. A (small) precise bound on the order of the maximal finite quotient was moreover given; in most cases the twin building lattice $\Lambda$ itself happens to be simple. The condition that $W$ is not virtually abelian was essential in loc.~cit., which relied on some weak hyperbolicity property of non-affine Coxeter groups. Rank two root systems were thus excluded since their Weyl group is infinite dihedral, hence virtually abelian (even though many rank two root systems are termed \emph{hyperbolic} within Kac--Moody theory).  

The goal  of this note is to provide a simplicity criterion applying  to that rank two case. Notice that when $\Phi$ has rank two, the twin building associated with $\Lambda$ is a twin tree $T_+ \times T_-$. Moreover $\Lambda$ is a lattice (then called a \textbf{twin tree lattice}) in $\Aut(T_+) \times \Aut(T_-)$ if and only if the root groups are finite; in other words, the condition on the order of the root groups ensuring that the covolume of $\Lambda$ is finite is automatically satisfied in this case. 

\begin{thm}
\label{thm - general}
Let $\Lambda$ be a group with a root group datum $(U_\alpha)_{\alpha \in \Phi}$ with finite root groups, indexed by a root system of rank $2$.  Suppose that $\Lambda$ is center-free and generated by the root groups. Assume moreover that the following conditions hold:
\begin{enumerate}[(i)]
\item There exist root groups $U_\phi, U_\psi$ associated with a prenilpotent pair of roots $\{\phi, \psi\}$ (possibly $\phi = \psi$) such that  the commutator $[U_\phi, U_\psi]$ is non-trivial. (Equivalently the maximal horospherical subgroups of $\Lambda$ are non-abelian.)

\item There is a constant $C>0$ such that for any prenilpotent pair of \emph{distinct} roots whose corresponding walls are at distance $\geqslant C$, the associated root groups commute.
\end{enumerate} 
Then the finitely generated group $\Lambda$ contains a  simple subgroup $\Lambda^0$ of finite index.

\end{thm}

We shall moreover see in Lemma~\ref{lem:Addendum} below that, with a little more information on the commutator $[U_\phi, U_\psi]$ in Condition~(i), the maximal finite quotient $\Lambda/\Lambda^0$ can be shown to be abelian of very small order.

As mentioned above, the main examples of twin building lattices arise from Kac--Moody theory. Specializing Theorem~\ref{thm - general} to that case, we obtain the following. 

\begin{thm}
\label{thm - KM}
Let $\Lambda$ be an adjoint split Kac--Moody group over the finite field $\FF_q$ and associated with the generalized Cartan matrix 
$A = \left( \begin{array}{cc} \hfill 2  & -m \\ -1 & \hfill 2 \end{array} \right)$, with $m>4$.

Then the commutator subgroup of $\Lambda$ is simple, has index~$\leqslant q$ in $\Lambda$, and acts as a non-uniform lattice on the product $T_+\times T_-$ of the twin trees associated with $\Lambda$. 
\end{thm}

The following consequence is immediate,  since split Kac--Moody groups over fields of order~$>3$ are known to be perfect. 
\begin{cor}
Let $\Lambda$ be an adjoint split Kac--Moody group over the finite field $\FF_q$ and associated with the generalized Cartan matrix 
$A = \left( \begin{array}{cc} \hfill 2  & -m \\ -1 & \hfill 2 \end{array} \right)$. 

If $m>4$ and $q > 3$, then $\Lambda$ is simple.
\end{cor}

Other examples of twin tree lattices satisfying the conditions from Theorem~\ref{thm - general} can be constructed in the realm of Kac--Moody theory, as \textbf{almost split groups}. Indeed, it is possible to construct non-split Kac--Moody groups of rank two, using Galois descent, so that the root groups are nilpotent of class two, while all commutation relations involving distinct roots are trivial.

Here is an example among many other possibilities.
Pick an integer $m \geqslant 2$ and consider the generalized Cartan matrix 
$A = \left( \begin{array}{ccc} \hfill 2 & \hfill -1 & \hfill -m \\ \hfill -1 & \hfill 2 & \hfill -m \\ \hfill -m & \hfill -m & \hfill 2 \end{array} \right)$.
This defines a split Kac-Moody group (over any field) whose Weyl group is the Coxeter group obtained, via Poincar\'e's theorem, from the tesselation of the hyperbolic plane by the (almost ideal) triangle with two vertices at infinity and one vertex of angle ${\pi \over 3}$.
The associated twinned buildings have apartments isomorphic to the latter hyperbolic tesselation. 
This Weyl group is generated by the reflections in the faces of the hyperbolic triangle, and the Dynkin diagram has a (unique, involutive) symmetry exchanging the vertices corresponding to the reflections in the two edges of the hyperbolic triangle meeting at the vertex of angle ${\pi \over 3}$.
Using \cite[Th. 2]{Morita} and \cite[Th. 1]{BP}, one sees that any prenilpotent pair of two roots leading to a non-trivial commutation relation between root groups is contained, up to conjugation by the Weyl group, in the standard residue of type $A_2$.

Suppose now that $\mathcal G_A(\FF_{q^2})$ is the split Kac--Moody group of that type, defined over a finite ground field of order $q^2$. According to  \cite[Proposition 13.2.3]{RemAst}, the non-trivial element of the Galois group of the extension $\FF_{q^2}/\FF_q$, composed with the symmetry of the Dynkin diagram, yields an involutory automorphism of $\mathcal G_A(\FF_{q^2})$  whose centraliser is  a quasi-split Kac-Moody group over the finite field $\FF_q$.
This quasi-split group acts on a twin tree obtained as the fixed point set of the involution acting on the twin building of the split group; the valencies are equal to $1+q$ and $1+q^3$, corresponding to root groups isomorphic to $(\FF_q,+)$ and to a $p$-Sylow subgroup of ${\rm SU}_3(q)$, respectively.
In particular, the root groups of order $q^3$ are nilpotent of step 2.
Moreover for any prenilpotent pair of two distinct roots, the corresponding root groups commute to one another: this follows from the last statement in the previous paragraph.
 
Further examples of twin tree lattices satisfying the simplicity criterion from Theorem~\ref{thm - general}, of a more exotic nature, can be constructed as in 
\cite{RemRon}. In particular it is possible that the two conjugacy classes of  root groups have coprime order. 

Finitely generated Kac--Moody groups associated with the generalized Cartan matrix 
$ \left( \begin{array}{cc} \hfill 2  & -4 \\ -1 & \hfill 2 \end{array} \right)$ or $ \left( \begin{array}{cc} \hfill 2  & -2 \\ -2 & \hfill 2 \end{array} \right)$, are known to be residually finite (and can in fact be identified with some $S$-arithmetic groups of positive characteristic). In particular it cannot be expected that the conclusions of Theorem~\ref{thm - KM} hold without any condition on the Cartan matrix $A$. The remaining open case is that of a matrix of the form $A_{m, n} =  \left( \begin{array}{cc} \hfill 2  & -m \\ -n & \hfill 2 \end{array} \right)$ with $m, n > 1$. In that case Condition~(ii) from Theorem~\ref{thm - general} holds, but Condition~(i) is violated. On the other hand, if the matrix $A_{m, n}$ is congruent to the matrix $A_{m', n'}$ modulo  $q-1$, then the corresponding Kac--Moody groups over $\FF_q$ are isomorphic (see \cite[Lem.~4.3]{Ca}). In particular if $(m', n') = (2, 2)$ or $(m', n') = (4, 1)$, then all these Kac--Moody groups are residually finite. It follows that over $\FF_2$, a rank two Kac--Moody group
 is either residually finite (because it is isomorphic to a Kac--Moody group of affine type), or virtually simple, by virtue of Theorem~\ref{thm - KM}. The problem whether this alternative holds for rank two Kac--Moody groups over larger fields remains open; its resolution will require to deal with Cartan matrices of the form $A_{m, n}$ with $m, n>1$.




\smallskip

\subsection*{Acknowledgements} The second author warmly thanks the organizers of the Special Quarter {\it Topology and Geometric Group Theory}~held at the Ohio State University (Spring 2011).

\section{Proof of the simplicity criterion}

\label{s - simplicity criterion}

Virtual simplicity will be established following the Burger--Mozes strategy from \cite{BurgerMozes} by combining the \textbf{Normal Subgroup Property}, abbreviated \textbf{(NSP)}, with the property of non-residual finiteness. This strategy was also implemented in \cite{CaRe}. The part of the work concerning (NSP) obtained in that earlier reference already included the rank two case, and thus applies to our current setting; its essential ingredient is the work of Bader--Shalom~\cite{BaderShalom}:

\begin{prop}\label{prop:NSP}
Let $\Lambda$ be a twin building lattice with associated root group datum $(U_\alpha)_{\alpha \in \Phi}$. Assume that $\Lambda$ is generated by the root groups. 

If $\Phi$ is irreducible, then every normal subgroup of $\Lambda$ is either finite central, or of finite index. In particular, if $\Lambda$ is center-free (equivalently if it acts faithfully on its twin building), then $\Lambda$ is just-infinite. 
\end{prop}

\begin{proof}
See \cite[Th.~18]{CaRe}. 
\end{proof}

The novelty in the present setting relies in the proof of non-residual finiteness. In the former paper \cite{CaRe}, we exploited some hyperbolic behaviour of non-affine Coxeter groups, appropriately combined with the commutation relations of $\Lambda$. This argument cannot be applied to infinite dihedral Weyl groups. Instead, we shall use the following non-residual finiteness result for wreath products, due to Meskin  \cite{Meskin}:

\begin{prop}\label{prop:Meskin}
Let $F, Z$ be two groups, and let $\Gamma = F \wr Z \cong (\bigoplus_{i \in {Z}} F) \rtimes {Z}$ be their wreath products. Assume that $Z$ is infinite. 

Then any finite index subgroup of $\Gamma$ contains the subgroup $\bigoplus_{i \in {Z}} [F, F]$. In particular, if $F$ is not abelian, then $\Gamma$ is not residually finite. 
\end{prop}

\begin{proof}
For each $i \in Z$, let $F_i$ be a copy of $F$, so that $F \wr Z =  (\bigoplus_{i \in {Z}} F_i) \rtimes Z$.

Let $\varphi \colon \Gamma \to Q$ be a homomorphism to a finite group $Q$. Since $Z$ is infinite, there is some $t \in Z \setminus \{ 1 \}$ 
such that $\varphi(t)=1$. 
Notice that, for all $i \in Z$ and all  $x \in F_i$, we have $txt^{-1}  \in F_{ti} \neq F_i$, whence $txt^{-1}  $ commutes with every element of $F_i$. Therefore, for all $y \in F_i$, we have
$$
\begin{array}{rcl}
\varphi\bigl([x,y]\bigr) & = & [\varphi(x),\varphi(y)] \\
& = & [\varphi(t x t^{-1}),\varphi(y)] \\
& = & \varphi\bigl([t xt^{-1},y]\bigr)\\
& = & 1.
\end{array}
$$
This proves that $[F_i, F_i]$ is contained in $\Ker(\varphi)$, and so is thus $\bigoplus_{i \in {Z}} [F_i, F_i]$. This proves that every finite index normal subgroup $\Gamma$ contains $\bigoplus_{i \in {Z}} [F_i, F_i]$. The desired result follows, since every finite index subgroup contains a finite index normal subgroup.
\end{proof}

\begin{proof}[Proof of Theorem~\ref{thm - general}]
Recall that in the case of twin trees, a pair of roots $\{ \phi; \psi \}$ is prenilpotent if and only if $\phi \supseteq \psi$ or $\psi \supseteq \phi$ (where the roots $\phi$ and $\psi$ are viewed as half-apartments).
By (i) there exists such a pair with $[U_\phi, U_\psi] \neq \{ 1 \}$ (possibly $\phi=\psi$). In particular the group $F = \langle U_\phi,U_\psi \rangle$ is non-abelian.

In view of (ii), the distance between the roots $\phi$ and $\psi$ in the trees on which $\Lambda$ acts is smaller than $C$. Pick an element in $\Lambda$ stabilising the standard twin apartment and acting on it as a translation of length~$>2C$. It follows from (ii) and from the axioms of Root Group Data that the subgroup  of $\Lambda$ generated 
by $F$ and $t$ is isomorphic to the wreath product $F \wr \ZZ$, where the cyclic factor is generated by $t$. 

Since $F$ is not abelian, we deduce from Proposition~\ref{prop:Meskin} that {$\Lambda$} contains a non-residually finite subgroup, and can therefore not be residually finite. On the other hand $\Lambda$ is just-infinite by Proposition~\ref{prop:NSP}. Therefore, we deduce from  \cite[Prop.~1]{Wilson} that the unique smallest finite index subgroup $\Lambda^0$ of $\Lambda$ is a finite direct product of $m \geqslant 1$ pairwise isomorphic simple groups. It remains to show that $m=1$. This follows from the fact that $\Lambda$ acts minimally (in fact: edge-transitively) on each half of its twin tree, and so does $\Lambda^0$. But a group acting minimally on a  tree cannot split non-trivially as a direct product. Hence $m=1$ and $\Lambda^0$ is a simple subgroup of finite index in $\Lambda$. 
\end{proof}

{ 
\begin{lem}\label{lem:Addendum}
Retain the hypotheses of Theorem~\ref{thm - general} and assume in addition that  one of the following conditions is satisfied:
\begin{itemize}
\item[(i$'$)] The  commutator $[U_\phi, U_\psi]$ contains some root group $U_\gamma$.

\item[(i$''$)] $\phi = \psi$, and if the rank one group $\la U_\phi, U_{-\phi} \ra$ is a sharply transitive group, then the commutator $[U_\phi, U_\phi]$ is of even order.  
\end{itemize} 
Then the maximal finite quotient $\Lambda/\Lambda^0$ afforded by Theorem~\ref{thm - general} is abelian. Moreover we have $|\Lambda/\Lambda^0| \leqslant \max_{\alpha \in \Phi} |U_\alpha|$ if (i$'$) holds, and $|\Lambda/\Lambda^0| \leqslant( \max_{\alpha \in \Phi} |U_\alpha/[U_\alpha, U_\alpha]|)^2$ if (i$''$) holds.
\end{lem}

\begin{proof}
Retain the notation from the proof of Theorem~\ref{thm - general}. Proposition~\ref{prop:Meskin} ensures that every finite index normal subgroup of $F \wr \ZZ$ contains the commutator subgroup $[F, F]$. In particular, so does $N = \Lambda^0$.

If $[U_\phi, U_\psi]$ contains some root group $U_\gamma$, then $U_\gamma$ is contained in $N$. Since $X_\gamma = \la U_\gamma, U_{-\gamma} \ra$ is a finite group acting $2$-transitively on the conjugacy class of $U_\gamma$, it follows that $X_\gamma$ is entirely contained in $N$. In particular, so is the element $r_\gamma \in X_\gamma$ acting as the reflection associated with $\gamma$ on the standard twin apartment.

Let now $\alpha \in \Phi$ be any root such that $\alpha \subset \gamma$ and that the wall $\partial \alpha$ is at distance $>C/2$ away from $\partial \gamma$. Then $\alpha \subset r_\gamma(-\alpha)$, and the walls associated with the latter two roots have distance~$>C$. By condition (ii), the corresponding root groups commute. 
Denoting by $\varphi \colon \Lambda \to \Lambda / N$ the quotient map,  we deduce 
$$
\begin{array}{rcl}
[\varphi(U_{\alpha}), \varphi(U_{-\alpha})]  & = &  [\varphi(U_{\alpha}), \varphi(U_{r_\gamma(-\alpha)})] \\
& = & \varphi\bigl([U_{\alpha}, U_{{r_\gamma(-\alpha)}}]\bigr)\\
& = & 1.
\end{array}
$$
Thus the image of the rank one group $X_\alpha = \la U_\alpha \cup U_{-\alpha} \ra$ under $\varphi$ is abelian, and identifies with a quotient of $U_\alpha$. 

Remark finally that there are only two conjugacy classes of root groups, the union of which generates the whole group $\Lambda$. One of these conjugacy classes has trivial image under $\varphi$, since $N$ contains the root group $U_\gamma$. The other conjugacy class contains roots $\alpha$ whose wall is far away from $\partial \gamma$. This implies that $\varphi(\Lambda) = \varphi(U_\alpha)$, which has been proved to be abelian. We are done in this case. 

Assume now that condition (i$''$) holds. Again, by Proposition~\ref{prop:Meskin}, the commutator $[U_\phi, U_\phi]$ is contained in $N$. 

If the rank one group $X_\phi = \la U_\phi, U_{-\phi} \ra$ is not a sharply $2$-transitive group, then it is a rank one simple group of Lie type and we may conclude that it is entirely contained in $N$. Hence the same argument as in the case (i$'$) with $\phi$ playing the role of $\gamma$ yields the conclusion. 

If the rank one group $X_\phi = \la U_\phi, U_{-\phi} \ra$ is a sharply $2$-transitive group, then we dispose of the additional hypothesis that the commutator $[U_\phi, U_\phi]$ contains an involution. Since $X_\phi$ is a  sharply $2$-transitive, all its involutions are conjugate. They must thus all be contained in $N$. In particular $N$ contains some involution $r_\phi$ swapping $U_\phi$ and $U_{-\phi}$. Again, this is enough to apply the same computation as above and conclude that for each root $\alpha$ whose wall is far away from $\partial \phi$, the image of $\la U_\alpha, U_{-\alpha}\ra$ under $\varphi$ is abelian and isomorphic to a quotient of $U_\alpha$. We take two distinct such roots $\alpha \subset \beta$ so that there is no root $\gamma$ strictly between $\alpha$ and $\beta$. Thus 
$U_\alpha$ and $U_\beta$ commute by the axioms of Root Group Data. Moreover $\Lambda$ is generated by $U_\alpha \cup U_{-\alpha} \cup U_\beta \cup U_{-\beta}$, and we have just seen that, modulo $N$, the root groups $U_\alpha$ and $U_{-\alpha}$ (resp. $U_\beta$ and $U_{-\beta}$) become equal, and abelian. It follows that $\Lambda/N$ is isomorphic to a quotient of the direct product $U_\alpha/[U_\alpha, U_\alpha] \times U_\beta/[U_\beta, U_\beta]$. The desired result follows. 
\end{proof}

\begin{remark}
Finite sharply $2$-transitive groups are all known; they correspond to finite near-fields, which were classified by Zassenhaus. All of them are either Dickson near-fields, or belong to a list of seven exceptional examples. An inspection of that list shows that, in all of these seven exceptions, the root group contains a copy of $\mathrm{SL}_2(\FF_3)$ or $\mathrm{SL}_2(\FF_5)$ (see \cite[\S 1.12]{Cameron}); in particular the commutator subgroup of a root group is always of even order in those cases. Thus condition (i$''$) from Lemma~\ref{lem:Addendum} only excludes certain sharply $2$-transitive groups associated with Dickson near-fields.
\end{remark}

}


\section{Kac--Moody groups of rank two}
\label{s - KM case}
 
Let $\Lambda$ be a Kac--Moody group over the finite field  $\FF_q$ of order $q$, associated with the generalized Cartan matrix $A_{m, n} =  \left( \begin{array}{cc} \hfill 2  & -m \\ -n & \hfill 2 \end{array} \right)$. The Weyl group of $\Lambda$ is the infinite dihedral group and $\Lambda$ is a twin tree lattice; the corresponding trees are both regular of degree $q+1$. 

When $mn< 4$, the matrix $A$ is of  \textbf{finite type} and $\Lambda$ is then a finite Chevalley group over $\FF_q$. 
When $mn=4$, the matrix $A$ is of \textbf{affine type} and $\Lambda$  is  linear, and even $S$-arithmetic; in particular it is residually finite.

In order to check that the conditions from Theorem~\ref{thm - general} are satisfied when $m>4$ and $n=1$, we need a sharp control on the commutation relations satisfied by the root groups. The key technical result is the following lemma, which follows from the work of Morita~\cite{Morita} and Billig--Pianzola~\cite{BP}.

\begin{lem}\label{lem:CommRel}
Let $\Pi = \{\alpha, \beta\}$ be the standard basis of the root system $\Delta$ for $\Lambda$ and set $t = r_\alpha r_\beta$. For all $i \in \ZZ$, let 
$\alpha_i = t^i \alpha$ and $\beta_i = t^i \beta$ and set
$$\Phi({+\infty}) = \{- \alpha_i, \beta_j \; | \; i, j \in \ZZ\}
\hspace{.5cm} \text{and } \hspace{.5cm}
\Phi({-\infty}) = \{ \alpha_i, -\beta_j \; | \; i, j \in \ZZ\}.$$

Assume that $m>4$ and $n=1$. Then for all $\phi, \psi \in\Phi(+\infty)$, either $U_\phi$ and $U_\psi$ commute, or we have 
$$
\{\phi, \psi\} = \{- \alpha_i, -\alpha_{i+1}\} \text{ for some } i \in \ZZ \text{ and } [U_\phi, U_\psi] = U_{\beta_i}.
$$ 
Similarly, for all  $\phi, \psi \in\Phi(-\infty)$,  either $U_\phi$ and $U_\psi$ commute, or we have 
$$
\{\phi, \psi\} = \{ \alpha_i, \alpha_{i+1}\} \text{ for some } i \in \ZZ \text{ and } [U_\phi, U_\psi] = U_{-\beta_i}.
$$ 

\end{lem}

\begin{proof}
Follows from Theorem~2 in \cite{Morita} and Theorem~1 in \cite{BP}. 
\end{proof}

\begin{proof}[Proof of Theorem~\ref{thm - KM}]
Lemma~\ref{lem:CommRel} readily implies that Conditions (i) and (ii) from Theorem~\ref{thm - general} are satisfied (we can take $C=2$ in this case), so that $\Lambda$ is virtually simple. {In fact, Lemma~\ref{lem:CommRel} shows that some root group is equal to the commutator of a pair of prenilpotent root groups, so that condition (i$'$) from Lemma~\ref{lem:Addendum} is satisfied. The latter ensures that $\Lambda^0$ is the commutator subgroup of $\Lambda$, and that the quotient $\Lambda/\Lambda^0$ is bounded above by the maximal order of a root group. Thus the theorem holds, since all the root groups have order $q$ in this case. }
\end{proof}

\end{document}